\documentclass[11pt,reqno]{amsart}
\usepackage{amssymb,amsmath,amsthm,amsfonts,latexsym}

\newtheorem{theorem}{Theorem}[section]
\newtheorem{lemma}[theorem]{Lemma}

\numberwithin{equation}{section}

\def \le{{\,\leqslant\,}}
\def \ge{{\,\geqslant\,}}

\begin{document}
\title{Waring-Goldbach problem for unlike powers}

\author{Zhenzhen Feng}
\address{School of Mathematics and Statistics, Minnan Normal University, Zhangzhou 363000, P. R. China}
\email{fengzz13@mails.jlu.edu.cn}

\author{Jing Ma*}
\address{School of Mathematics, Jilin University, Changchun 130012, P. R. China}
\email{jma@jlu.edu.cn}

\thanks{*corresponding author}

\begin{abstract}
In this paper, we investigate exceptional sets in the Waring-Goldbach problem for unlike powers.
For example, estimates are obtained for sufficiently large integers below a parameter subject to the necessary local conditions
that do not have a representation as the sum of
a square of prime, a cube of prime and a sixth power of prime and a $k$-th power of prime. These results improve the
recent result due to Br\"udern in the order of magnitude.
Furthermore, the method can be also applied to the similar estimates for the exceptional sets for Waring-Goldbach
problem for unlike powers.
\end{abstract}

\maketitle


\section{Introduction}
Let $N$, $k_1$, $k_2$, $\cdots$, $k_r$ be natural numbers such that $2 \leq k_1\leq k_2 \leq \cdots \leq k_r$.
The Waring-Goldbach problem for unlike powers concerns the representation of $N$ as the form
$$
N=p_1^{k_1}+p_2^{k_2}+\cdots+p_r^{k_r}.
$$
Not very much is known about results of this kind. However, these topics have attracted mathematicians' attentions.

Schwarz \cite{Schwarz} considered the exceptional set of expressing an positive even number as the sum of a square of prime,
a cube of prime, a sixth power of prime and a $k$-th power of prime, i.e.
\begin{align}\label{236k}
n=p_1^2+p_2^3+p_3^6+p_4^k,
\end{align}
where $p_1$, $p_2$, $p_3$, $p_4$ are primes. Let $E_1(k, N)$ be the number of positive even integers $n$ up to $N$ which can not be written in
the form (\ref{236k}). Exactly, Schwarz \cite{Schwarz} showed that $E_1(k, N)\ll N(\log N)^{-A}$ for any fixed $A>0$.
Recently, Br\"udern \cite{Brudern} improved this result and
established that $E_1(k, N)\ll N^{1-\frac{1}{8k^2}+\varepsilon}$. In this paper, we further improve the result of
Br\"udern by giving
\begin{theorem}\label{Thm236k}
Let $E_1(k, N)$ be defined as above. We have
$$
E_1(k, N)\ll N^{1-\theta_1(k)+\varepsilon},
$$
among which,
\begin{align*}
\theta_1(k)=
\begin{cases}
 \frac{1}{54}, & k=6,\\
 \frac{1}{81}, & k=7,\\
 \frac{1}{54x}, & k\ge 8,
\end{cases}
\end{align*}
where
\begin{align}\label{236kx}
x=
\begin{cases}
\left\lceil(\frac{k}{6}+1-[\frac{k}{6}])2^{[\frac{k}{6}]-1}\right\rceil, &  8\le k\le 23,\\
\left\lceil \frac{7k}{6}-20 \right\rceil, &  24\le k\le 29,\\
\left\lceil (\frac{k}{6}-\frac{1}{2}[\frac{k}{6}])([\frac{k}{6}]+1)\right\rceil, &  k\ge 30.
\end{cases}
\end{align}
Here $\lceil a \rceil$ means the smallest integer no smaller than $a$ and $[a]$ means the biggest integer no bigger than $a$.
\end{theorem}

\noindent\textit{Remark:} We can compare the results of Theorem 1.1 with those of Br\"udern \cite{Brudern}. For example, we obtain
$E_1(6, N)\ll N^{1-\frac{1}{54}+\varepsilon}$ and $E_1(7, N)\ll N^{1-\frac{1}{81}+\varepsilon}$. Meanwhile, Br\"udern's results
indicated that $E_1(6, N)\ll N^{1-\frac{1}{288}+\varepsilon}$ and $E_1(7, N)\ll N^{1-\frac{1}{392}+\varepsilon}$.
In addition, for large value $k$, Theorem 1.1 gives that $E_1(k, N)\ll N^{1-\frac{1}{\frac{3}{4}k^2+O(k)}+\varepsilon}$,
whereas Br\"udern's result \cite{Brudern} showed that $E_1(k, N)\ll N^{1-\frac{1}{8k^2}+\varepsilon}$.

In the same paper \cite{Schwarz}, Schwarz also considered the problem of representing a large even integer $n$ in the form
\begin{align}\label{244k}
n=p_1^2+p_2^4+p_3^4+p_4^k,
\end{align}
where $p_1, p_2, p_3, p_4$ are primes.
Let $E_2(k, N)$ denote the number of positive even integers $n$ up to $N$ which can not be written in the
form (\ref{244k}). In fact, Schwarz \cite{Schwarz} proved that $E_2(k, N)\ll N(\log N)^{-A}$ for any fixed $A>0$.
Using the similar method to treat Theorem \ref{Thm236k}, we obtain the following result:
\begin{theorem}\label{Thm244k}
Let $E_2(k, N)$ be defined as above. We have
$$
E_2(k, N)\ll N^{1-\theta_2(k)+\varepsilon},
$$
here
\begin{align*}
\theta_2(k)=
\begin{cases}
 \frac{1}{32}, & k=4,
\\ \frac{1}{48}, & k=5,
\\ \frac{1}{64}, & 6\le k\le 8,
\\ \frac{1}{48x}, & k\ge 9,
\end{cases}
\end{align*}
where
\begin{align*}
x=
\begin{cases}
\left\lceil (\frac{k}{4}+1-[\frac{k}{4}])2^{[\frac{k}{4}]-1}\right\rceil, &  9\le k\le 19,\\
\left\lceil (\frac{k}{4}-\frac{1}{2}[\frac{k}{4}])([\frac{k}{4}]+1)\right\rceil, &  k\ge 20.
\end{cases}
\end{align*}
\end{theorem}
\noindent\textit{Remark:} For example, we obtain that
$E_2(4, N)\ll N^{1-\frac{1}{32}+\varepsilon}$ and $E_2(6, N)\ll N^{1-\frac{1}{64}+\varepsilon}$. Meanwhile, Br\"udern's method in \cite{Brudern}
indicated that $E_2(4, N)\ll N^{1-\frac{1}{128}+\varepsilon}$ and $E_2(6, N)\ll N^{1-\frac{1}{288}+\varepsilon}$.
In addition, for large value $k$, Theorem \ref{Thm244k} gives that $E_2(k, N)\ll N^{1-\frac{1}{\frac{3}{2}k^2+O(k)}+\varepsilon}$,
whereas Br\"udern's method in \cite{Brudern} showed that $E_2(k, N)\ll N^{1-\frac{1}{8k^2}+\varepsilon}$.

Another related problem is to study for the diophantine equation
\begin{align}\label{235k}
n=p_1^2+p_2^3+p_3^5+p_4^k
\end{align}
where $p_1$, $p_2$, $p_3$, $p_4$ are primes. Let $E_3(k, N)$ be the number of even integers $n\le N$ that can not
be represented in the form (\ref{235k}). In 1953, Prachar \cite{Prachar} proved that
$E_3(4, N)\ll N(\log N)^{-\frac{30}{47}+\varepsilon}$. This has been improved by a number of authors (c. f.\cite{Bauer1,Bauer2,RenTsang1,RenTsang2}).
The latest result is
\begin{align*}
E_3(4, N)\ll N^{1-\frac{1}{16}+\varepsilon}
\end{align*}
given by Zhao \cite{Zhao3}. For general $k\ge 5$, Lu and Shan \cite{LuShan} proved that $E_3(k, N)\ll N(\log N)^{-c}$ for
some $c>0$. Lately, it was improved to $E_3(k, N)\ll N^{1-\frac{1}{3k\times 2^{k-2}}+\varepsilon}$ by Liu \cite{LiuZ}. The
current best result was given by Hoffman and Yu \cite{HoffmanYu} which is
\begin{align}\label{E3HY}
E_3(k, N)\ll N^{1-\frac{47}{420\cdot 2^s}+\varepsilon}
\end{align}
where $s=[\frac{k+1}{2}]$. In this paper, we established the following result which improves (\ref{E3HY}).
\begin{theorem}\label{Thm235k}
Let $E_3(k, N)$ be defined as above. We have
$$
E_3(k, N)\ll N^{1-\theta_3(k)+\varepsilon},
$$
here
\begin{align*}
\theta_3(k)=
\begin{cases}
 \frac{1}{24}, & k=5,
\\ \frac{2}{81}, & k=6,
\\ \frac{1}{36x}, & k\ge 7,
\end{cases}
\end{align*}
where
\begin{align*}
x=
\begin{cases}
\left\lceil (\frac{k}{6}+1-[\frac{k}{6}])2^{[\frac{k}{6}]-1}\right\rceil, &  7\le k\le 23,
\\ \left\lceil \frac{7k}{6}-20 \right\rceil, &  24\le k\le 29,
\\ \left\lceil (\frac{k}{6}-\frac{1}{2}[\frac{k}{6}])([\frac{k}{6}]+1)\right\rceil, &  k\ge 30.
\end{cases}
\end{align*}
\end{theorem}
\noindent\textit{Remark:} Our results indeed improve the result of Hoffman and Yu \cite{HoffmanYu}. For example, we
obtain that $E_3(5, N)\ll N^{1-\frac{1}{24}+\varepsilon}$ and $E_3(7, N)\ll N^{1-\frac{1}{72}+\varepsilon}$. Meanwhile,
Hoffman and Yu's results indicated that $E_3(5, N)\ll N^{1-\frac{47}{3360}+\varepsilon}$
and $E_3(7, N)\ll N^{1-\frac{47}{6720}+\varepsilon}$. In addition, for large value $k$, $E_3(k, N)\ll N^{1-\frac{1}{\theta(k)}+\varepsilon}$,
Hoffman and Yu \cite{HoffmanYu} showed that $\theta(k)$ grows exponentially, whereas, Theorem \ref{Thm235k} implicates that
$\theta(k)=\frac{k^2}{2}+O(k)$ with polynomial growth.

Finally, we consider the problem of representing a large odd integer $n$ in the form
\begin{align}\label{3333k}
n=p_1^3+p_2^3+p_3^3+p_4^3+p_5^k,
\end{align}
where $p_1$, $p_2$, $p_3$, $p_4$ and $p_5$ are primes. Let $E_4(k, N)$ denote the number of positive odd integers $n$ up to $N$
which can not be written in the form (\ref{3333k}). In the following result, we will give a up bound for $E_4(k, N)$ for $k\ge 4$.
\begin{theorem}\label{Thm3333k}
Let $E_4(k, N)$ be defined as above. We have
$$
E_4(k, N)\ll N^{1-\theta_4(k)+\varepsilon},
$$
here
\begin{align*}
\theta_4(k)=
\begin{cases}
 \frac{1}{24}, & k=4,
\\ \frac{1}{54}, & k=5,
\\ \frac{1}{9x}, & k\ge 6,
\end{cases}
\end{align*}
where
\begin{align*}
x=\begin{cases}
\left\lceil \frac{14k}{3}-20 \right\rceil, &  k=6,7,
\\ \left\lceil (\frac{2k}{3}-\frac{1}{2}[\frac{2k}{3}])([\frac{2k}{3}]+1)\right\rceil, &  k\ge 8.
\end{cases}
\end{align*}
\end{theorem}
As usual, we abbreviate $e^{2\pi i\alpha}$ to $e(\alpha)$. The letter $p$, with or without indices, is
prime number. The letter $\varepsilon$ denotes a sufficiently small positive real number, and the value of $\varepsilon$ may
change from statement to statement. Let $N$ be a real number sufficiently large in terms of $\varepsilon$ and $k$. We
use $\ll$ and $\gg$ to denote Vinogradov's well-know notation, while implied constant may depend on $\varepsilon$ and $k$.

\section{Preliminaries and lemmas}
We will prove Theorems 1.1-1.4 by using circle method. Now the treatment for major arcs of Hardy-Littlewood method are standard nowadays,
for example  Liu and Zhan \cite{LiuZhan}. We need the following lemmas to control the minor arcs of circle method.
\begin{lemma}\label{LemHuaBou}
Let
$$S_k(\alpha)=\sum_{N/4<p^k\le N}(\log p)e(\alpha p^k).$$
Then for $1\le j\le k$, we have
$$\int_0^1\big|S_k^{2^j}(\alpha)\big|d\alpha\ll N^{\frac{1}{k}(2^j-j)+\varepsilon}$$
and
$$\int_0^1\big|S_k^{j(j+1)}(\alpha)\big|d\alpha\ll N^{\frac{j^2}{k}+\varepsilon}.$$
\end{lemma}
In fact, Lemma \ref{LemHuaBou} is the classical result of Hua \cite{Hua} and the recent work of Bourgain \cite{Bourgain}. The
next lemma is a generalization of Lemma {\ref{LemHuaBou}}.

\begin{lemma}\label{lemmain}
Let $S_k(\alpha)$ be defined as Lemma \ref{LemHuaBou}. For $0<\delta\le 1$,
\begin{align*}
\int_0^1\big|S_k^{2x}(\alpha)\big|d\alpha\ll N^{\frac{2x}{k}-\delta+\varepsilon},
\end{align*}
where
\begin{align}\label{mainx}
x=
\begin{cases}
\left\lceil (k\delta+1-[k\delta])2^{[k\delta]-1}\right\rceil, &  [k\delta]\le 3,
\\ \left\lceil 7k\delta-20 \right\rceil, & [k\delta]=4,
\\ \left\lceil (k\delta-\frac{1}{2}[k\delta])([k\delta]+1)\right\rceil, &  [k\delta]\ge 5.
\end{cases}
\end{align}
\end{lemma}
\begin{proof}
For $\delta=1$, this is Lemma \ref{LemHuaBou}.
Next, we consider the case $0<\delta<1$: \\
for $[k\delta]\le 3$, clearly by (\ref{mainx}) we have
$$2^{[k\delta]}\le 2x\le 2^{[k\delta]+1}.$$
Applying H\"older's inequality and Hua's lemma, one has
\begin{align*}
\int_0^1\big|S_k^{2x}(\alpha)\big|d\alpha\ll & \Big(\int_0^1\big|S_k^{2^{[k\delta]}}(\alpha)\big|d\alpha\Big)^a
\Big(\int_0^1\big|S_k^{2^{[k\delta]+1}}(\alpha)\big|d\alpha\Big)^b
\notag
\\
\ll & N^{\frac{2x}{k}-c+\varepsilon},
\end{align*}
where
$$
a=2-\frac{x}{2^{[k\delta]-1}}, \quad b=\frac{x}{2^{[k\delta]-1}}-1, \quad
c=\frac{[k\delta]+\frac{2x}{2^{[k\delta]}}-1}{k}.$$
Recall that $x\ge (k\delta+1-[k\delta])2^{[k\delta]-1}$, so we have $c\ge \delta$. Thus this lemma holds for
$[k\delta]\le 3$.

For $[k\delta]=4$, obviously by (\ref{mainx}) we have
$$16 <2x\le 30.$$
Applying H\"older's inequality and Lemma \ref{LemHuaBou}, one has
\begin{align*}
\int_0^1\big|S_k^{2x}(\alpha)\big|d\alpha\ll & \Big(\int_0^1\big|S_k^{16}(\alpha)\big|d\alpha\Big)^{\frac{15}{7}-\frac{x}{7}}
\Big(\int_0^1\big|S_k^{30}(\alpha)\big|d\alpha\Big)^{\frac{x}{7}-\frac{8}{7}}
\notag
\\ \ll & N^{\frac{2x}{k}-\frac{x+20}{7k}+\varepsilon}.
\end{align*}
This combining with $x\ge 7k\delta-20$ gives $\frac{x+20}{7k}\ge \delta.$
Thus this lemma holds for
$[k\delta]=4$.

For $[k\delta]\ge 5$, by (\ref{mainx}) we have
$$[k\delta]([k\delta]+1)\le 2x\le ([k\delta]+1)([k\delta]+2).$$
Applying H\"older's inequality and Lemma \ref{LemHuaBou}, one has
\begin{align*}
\int_0^1\big|S_k^{2x}(\alpha)\big|d\alpha\ll & \Big(\int_0^1\big|S_k^{[k\delta]([k\delta]+1)}(\alpha)\big|d\alpha\Big)^{a}
\Big(\int_0^1\big|S_k^{([k\delta]+1)([k\delta]+2)}(\alpha)\big|d\alpha\Big)^{b}
\notag
\\ \ll & N^{\frac{2x}{k}-c+\varepsilon},
\end{align*}
where
$$
a=1-\frac{x}{[k\delta]+1}+\frac{[k\delta]}{2}, \quad
b=\frac{x}{[k\delta]+1}-\frac{[k\delta]}{2},\quad
c=\frac{\frac{x}{[k\delta]+1}+\frac{[k\delta]}{2}}{k}.
$$
Then we have $c\ge \delta$ because of $x\ge (k\delta-\frac{[k\delta]}{2})([k\delta]+1)$.
This lemma holds for
$[k\delta]\ge 5$.
Hence, this lemma holds for $0<\delta\le 1$.
\end{proof}

\begin{lemma}\label{lemS2x}
For $k\ge 3$, we have
\begin{align*}
\int_0^1|S_2^2(\alpha)S_k^{2x}(\alpha)|d\alpha\ll N^{\frac{2x}{k}+\varepsilon},
\end{align*}
where
\begin{align*}
x=
\begin{cases}
\left\lceil (\frac{k}{2}+1-[\frac{k}{2}])2^{[\frac{k}{2}]-1}\right\rceil, &  3\le k\le 9,
\\ \left\lceil (\frac{k}{2}-\frac{1}{2}[\frac{k}{2}])([\frac{k}{2}]+1)\right\rceil, &  k\ge 10.
\end{cases}
\end{align*}
\end{lemma}
\begin{proof}
$\int_0^1|S_2^2(\alpha)S_k^{2x}(\alpha)|d\alpha$ is no more than $N^{\varepsilon}$ times the number of solutions of the equation
$$t_1^2-t_2^2=y_1^k+y_2^k+\cdot\cdot\cdot+y_x^k-y_{x+1}^k-\cdot\cdot\cdot-y_{2x}^k$$
with $N^{\frac{1}{2}}<t_1,t_2\le 2N^{\frac{1}{2}}$ and $N^{\frac{1}{k}}<y_1,y_2,\cdot\cdot\cdot,y_{2x}\le 2N^{\frac{1}{k}}$. If $t_1\neq t_2$, the contribution is bounded by $N^{\frac{2x}{k}+\varepsilon}$.
If $t_1=t_2$, the contribution is bounded by $N^{\frac{1}{2}+\varepsilon}\int_0^1|S_k^{2x}|d\alpha$. Thus
$$\int_0^1|S_2^2(\alpha)S_k^{2x}(\alpha)|d\alpha\ll N^{\frac{2x}{k}+\varepsilon}+N^{\frac{1}{2}+\varepsilon}\int_0^1|S_k^{2x}(\alpha)|d\alpha.$$
What we need is
$$\int_0^1|S_k^{2x}(\alpha)|d\alpha\ll N^{\frac{2x}{k}-\frac{1}{2}+\varepsilon}.$$
Hence this lemma holds by Lemma \ref{lemmain} with $\delta=\frac{1}{2}$.
\end{proof}

\begin{lemma}\label{lemS24x}
For $k\ge 4$, we have
\begin{align*}
\int_0^1|S_2^2(\alpha)S_4^2(\alpha)S_k^{2x}(\alpha)|d\alpha\ll N^{\frac{2x}{k}+\frac{1}{2}+\varepsilon},
\end{align*}
where
\begin{align*}
x=\begin{cases}
\left\lceil (\frac{k}{4}+1-[\frac{k}{4}])2^{[\frac{k}{4}]-1}\right\rceil, &  4\le k\le 19,
\\ \left\lceil (\frac{k}{4}-\frac{1}{2}[\frac{k}{4}])([\frac{k}{4}]+1)\right\rceil, &  k\ge 20.
\end{cases}
\end{align*}
\end{lemma}
\begin{proof}
$\int_0^1|S_2^2(\alpha)S_4^2(\alpha)S_k^{2x}(\alpha)|d\alpha$ is no more than $N^{\varepsilon}$ times the number of solutions for the equation
$$t_1^2-t_2^2=y_1^4-y_2^4+z_1^k+z_2^k+\cdot\cdot\cdot+z_x^k-z_{x+1}^k-\cdot\cdot\cdot-z_{2x}^k$$
with $P_2< t_1,t_2\le 2P_2$, $P_4< y_1,y_2\le 2P_4$ and $P_k< z_1,z_2,\cdot\cdot\cdot,z_{2x}\le 2P_k$, where $N/4<P_2^2, P_4^4, P_k^k\le N.$
If $t_1\neq t_2$, the contribution is bounded by $P_4^{2+\varepsilon}P_k^{2x}$. If $t_1=t_2, y_1\neq y_2$, the contribution is
bounded by $P_2P_k^{2x+\varepsilon}$. If $t_1=t_2,   y_1= y_2$, the contribution is bounded by $P_2^{1+\varepsilon}P_4\int_0^1|S_k(\alpha)|^{2x}d\alpha$.
Thus
\begin{align*}
\int_0^1|S_2^2(\alpha)S_4^2(\alpha)S_k^{2x}(\alpha)|d\alpha\ll N^{\varepsilon}P_4^{2+\varepsilon}P_k^{2x}+N^{\varepsilon}P_2P_4\int_0^1|S_k(\alpha)|^{2x}d\alpha.
\end{align*}
What we need is
$$\int_0^1|S_k^{2x}(\alpha)|d\alpha\ll N^{\frac{2x}{k}-\frac{1}{4}+\varepsilon}.$$
Hence this lemma holds by Lemma \ref{lemmain} with $\delta=\frac{1}{4}$.
\end{proof}

\begin{lemma}\label{lemS23k}
For $k\ge 3$, we have
\begin{align*}
\int_0^1|S_2^2(\alpha)S_3^2(\alpha)S_k^{2x}(\alpha)|d\alpha\ll N^{\frac{2x}{k}+\frac{2}{3}+\varepsilon},
\end{align*}
where
\begin{align*}
x=\begin{cases}
\left\lceil (\frac{k}{6}+1-[\frac{k}{6}])2^{[\frac{k}{6}]-1}\right\rceil, &  3\le k\le 23,\\
\left\lceil \frac{7k}{6}-20 \right\rceil, &  24\le k\le 29,\\
\left\lceil (\frac{k}{6}-\frac{1}{2}[\frac{k}{6}])([\frac{k}{6}]+1)\right\rceil, &  k\ge 30.
\end{cases}
\end{align*}
\end{lemma}
\begin{proof}
The proof is similar as the proof of Lemma \ref{lemS24x} with $\delta=\frac{1}{6}$.
\end{proof}

\begin{lemma}\label{8233}
For $k\ge 3$, we have
\begin{align*}
\int_0^1|S_3^4(\alpha)S_k^{2x}(\alpha)|d\alpha\ll N^{\frac{2x}{k}+\frac{1}{3}+\varepsilon},
\end{align*}
where
\begin{align*}
x=\begin{cases}
\left\lceil (\frac{2k}{3}+1-[\frac{2k}{3}])2^{[\frac{2k}{3}]-1} \right\rceil, &  3\leq k \leq 5, \\
\left\lceil \frac{14k}{3}-20 \right\rceil, &  k=6,7,\\
\left\lceil (\frac{2k}{3}-\frac{1}{2}[\frac{2k}{3}])([\frac{2k}{3}]+1)\right\rceil, &  k\ge 8.
\end{cases}
\end{align*}
\end{lemma}

\begin{proof}
We have
\begin{align}\label{4z1}
\int_0^1|S_3^4(\alpha)S_k^{2x}(\alpha)|d\alpha\ll N^{\varepsilon}\int_0^1|f_3^4(\alpha)S_k^{2x}(\alpha)|d\alpha,
\end{align}
where $$f_3(\alpha)=\sum_{t\sim P_3}e(\alpha t^3).$$
By Lemma 2.3 in \cite{Vaughan}, one has
\begin{align*}
\big|f_3(\alpha)\big|^4\ll P_3\sum_{|h_1|< P_3}\sum_{|h_2|< P_3}\sum_{x\in \mathcal{J}}e(\alpha\Delta(t^3; {\bf h})),
\end{align*}
where $\mathcal{J}=\mathcal{J}({\bf h})$ is a subinterval of $[P_3, 2P_3)$ and $\Delta(t^3; {\bf h})$ is the second-order forward
difference of the function $t\rightarrow t^3$ with steps $h_1, h_2$, that is,
$$\Delta(t^3; {\bf h})=3h_1h_2(2t+h_1+h_2).$$
Thus, we deduce from (\ref{4z1}) that
\begin{align*}
\int_0^1|S_3^4(\alpha)S_k^{2x}(\alpha)|d\alpha\ll P_3J(P_3),
\end{align*}
where $J(P_3)$ is the number of solutions of the diophantine equation
\begin{align}\label{4z2}
\Delta(t^3; {\bf h})=3h_1h_2(2t+h_1+h_2)=p_1^k+p_2^k+\cdot\cdot\cdot+p_x^k-q_1^k-q_2^k-\cdot\cdot\cdot-q_x^k
\end{align}
subject to
\begin{align}\label{4z3}
P_3\le t\le 2P_3, \quad |h_i|<P_3, \quad P_k<p_1,\cdot\cdot\cdot, p_x,q_1,\cdot\cdot\cdot, q_x\le 2P_k, \quad N/4<P_3^3, P_k^k\le N.
\end{align}
The number of solutions of (\ref{4z2}), (\ref{4z3}) with $\Delta(t^3; {\bf h})=0$ is bounded by\\
$P_3^{2+\varepsilon}\int_0^1\big|S_k^{2x}(\alpha)\big|d\alpha$.
The number of solutions of (\ref{4z2}), (\ref{4z3}) with $\Delta(t^3; {\bf h})\neq 0$ is bounded by
$N^{\frac{2x}{k}+\varepsilon}$.Thus,
\begin{align*}
\int_0^1|S_3^4(\alpha)S_k^{2x}(\alpha)|d\alpha\ll N^{1+\varepsilon}\int_0^1\big|S_k^{2x}(\alpha)\big|d\alpha+N^{\frac{1}{3}+\frac{2x}{k}+\varepsilon}.
\end{align*}
Thus, we just need
$$
\int_0^1|S_k^{2x}(\alpha)|d\alpha\ll N^{\frac{2x}{k}-\frac{2}{3}+\varepsilon}.
$$
Hence it establishes this  lemma by Lemma \ref{lemmain} with $\delta=\frac{2}{3}$.
\end{proof}

 \vskip3mm

\bigskip

\section{Proof of Theorem \ref{Thm236k}}
The purpose of this section is to concentrate on proving Theorem \ref{Thm236k}. We establish Theorem \ref{Thm236k} by means of the Hardy-Littlewood
method. We will give the proof of Theorem \ref{Thm244k} in Sect. 4 and will describe the straight forward modifications needed for
Theorem \ref{Thm235k} in Sect. 5. In Sect. 6, we will give the outline of proof of Theorem \ref{Thm3333k}.

Let $S_k(\alpha)$ be defined as in Lemma \ref{LemHuaBou}.
We denote
\begin{align}\label{r236k}
r(n)=\sum_{\substack{p_1^2+p_2^3+p_3^6+p_4^k=n\\ N/4<p_1^2,p_2^3,p_3^6,p_4^k\le N}}(\log p_1)(\log p_2)
(\log p_3)(\log p_4),
\end{align}
where $p_1$, $p_2$, $p_3$, $p_4$ are primes.
Let $Q=N^{\frac{2}{5k}}$, and write $\mathfrak{M}(Q)$ for the union of the intervals
$$\{\alpha\in [0,1]: |q\alpha-a|\le QN^{-1}\}$$
with $1\le a\le q, (a,q)=1$ and $1\le q\le Q$. We define $ \mathfrak{M}=\mathfrak{M}(Q), \textrm{  }\ \mathfrak{m}=[0,1]\backslash\mathfrak{M}$. Thus the formula
(\ref{r236k}) becomes
$$r(n)=\Big\{\int_{\mathfrak{M}}+\int_{\mathfrak{m}}\Big\}S_2(\alpha)S_3(\alpha)S_6(\alpha)S_k(\alpha)e(-n\alpha)d\alpha.$$
Whenever $\mathfrak{B}\subset [0,1]$ is measurable, we put
$$r_{\mathfrak{B}}(n, N)=\int_{\mathfrak{B}}S_2(\alpha)S_3(\alpha)S_6(\alpha)S_k(\alpha)e(-n\alpha)d\alpha.$$
Then we have
$$r(n)=r_{[0,1]}(n, N)=\int_ 0^1S_2(\alpha)S_3(\alpha)S_6(\alpha)S_k(\alpha)e(-n\alpha)d\alpha.$$
Next we will deal with the integral of major arcs and minor arcs respectively. Applying the now standard methods of
enlarging major arcs (c.f. \cite{LiuZhan}), we can get the following result:
\begin{lemma}
For all even integer $n$ with $N<n \le 2N$, one has $r_{\mathfrak{M}}(n, N)\gg N^{\frac{1}{k}}$.
\end{lemma}
To estimate the integral of the minor arcs, we split the minor arcs in two part. Let $1\le Y\le N^{\frac{1}{8}}$, and denote $\mathfrak{N}$
the union of the pairwise disjoint intervals
$$\mathfrak{N}_{q,a}(Y)=\{\alpha\in[0,1]: |q\alpha-a|\le Y/N\}$$
with $1\le a\le q, (a,q)=1$ and $1\le q\le Y$. We write $\mathfrak{N}=\mathfrak{N}(N^{\frac{1}{8}})$ and $\mathfrak{n}=\mathfrak{m}\backslash \mathfrak{N}.$
\begin{lemma}\label{nS2S3S5}
For $\alpha\in \mathfrak{n}$, we have
$$
S_2(\alpha)\ll N^{\frac{1}{2}-\frac{1}{16}+\varepsilon}, \quad
S_3(\alpha)\ll N^{\frac{1}{3}-\frac{1}{36}+\varepsilon}, \quad
S_4(\alpha)\ll N^{\frac{1}{4}-\frac{1}{96}+\varepsilon}.
$$
\end{lemma}
\begin{proof}
For any given $\alpha\in \mathbb{N}$, by Dirichlet's approximation theorem, there exist $a\in \mathbb{Z}$ and $q\in \mathbb{N}$ with
$$(a,q)=1, 1\le q\le N^{\frac{5}{12}} \textrm{ and }\ |q\alpha-a|\le N^{-\frac{5}{12}}.$$
Then by Theorem 1 in \cite{Kumchev}, one has
$$
S_2(\alpha)\ll N^{\frac{1}{2}-\frac{1}{16}+\varepsilon}+\frac{N^{\frac{1}{2}+\varepsilon}}{(q+N|q\alpha-a|)^{\frac{1}{2}}},
$$
and
$$
S_4(\alpha)\ll N^{\frac{1}{4}-\frac{1}{96}+\varepsilon}+\frac{N^{\frac{1}{4}+\varepsilon}}{(q+N|q\alpha-a|)^{\frac{1}{2}}},
$$
and by Lemma 2.3 in \cite{Zhao2}, one has
$$S_3(\alpha)\ll N^{\frac{1}{3}-\frac{1}{36}+\varepsilon}+\frac{N^{\frac{1}{3}+\varepsilon}}{(q+N|q\alpha-a|)^{\frac{1}{2}}}.$$
If $\alpha\in \mathfrak{n}$, then
$$q> N^{\frac{1}{8}} \textrm{ or }\ q\le N^{\frac{1}{8}}, \textrm{  }\ N^{-\frac{7}{8}}\le |q\alpha-a|<N^{-\frac{5}{12}}.$$
In any case, we have
$$q+|q\alpha-a|\gg N^{\frac{1}{8}},$$
then this lemma clearly holds.
\end{proof}
\noindent\textit{Proof of Theorem 1.1.}
 By Bessel's inequality, we have
\begin{align*}
\sum_{N< n \le 2N} & \Big|\int_{\mathfrak{m}}S_2(\alpha)S_3(\alpha)S_6(\alpha)S_k(\alpha)e(-n\alpha)d\alpha\Big|^2
\le
\int_{\mathfrak{m}}\big|S_2^2(\alpha)S_3^2(\alpha)S_6^2(\alpha)S_k^2(\alpha)\big|d\alpha.
\end{align*}
To prove Theorem \ref{Thm236k}, it suffices to show that
\begin{align}\label{thm1236k}
\int_{\mathfrak{m}}\big|S_2^2(\alpha)S_3^2(\alpha)S_6^2(\alpha)S_k^2(\alpha)\big|d\alpha\ll N^{1+\frac{2}{k}-\theta_1(k)+\varepsilon},
\end{align}
where
$\theta_1(k)$ is defined in Theorem \ref{Thm236k}.

Obviously, we know that
\begin{align*}
\int_{\mathfrak{m}}\big|S_2^2(\alpha)S_3^2(\alpha)S_6^2(\alpha)S_k^2(\alpha)\big|d\alpha\ll &
\int_{\mathfrak{N}\backslash\mathfrak{M}}\big|S_2^2(\alpha)S_3^2(\alpha)S_6^2(\alpha)S_k^2(\alpha)\big|d\alpha
\\
 &+ \int_{\mathfrak{n}}\big|S_2^2(\alpha)S_3^2(\alpha)S_6^2(\alpha)S_k^2(\alpha)\big|d\alpha.
\end{align*}
By the estimate on Page 80 in \cite{Brudern}, one has
\begin{align}\label{thm1236km}
\int_{\mathfrak{N}\backslash\mathfrak{M}}\big|S_2^2(\alpha)S_3^2(\alpha)S_6^2(\alpha)S_k^2(\alpha)\big|d\alpha\ll
N^{1+\frac{2}{k}-\frac{1}{4k}+\varepsilon} \ll N^{1+\frac{2}{k}-\theta_1(k)+\varepsilon}
\end{align}
since $\frac{1}{4k}>\theta_1(k)$ for all $k\ge 6$.

Next we estimate $\int_{\mathfrak{n}}\big|S_2^2(\alpha)S_3^2(\alpha)S_6^2(\alpha)S_k^2(\alpha)\big|d\alpha$.

For $k=6$, by Lemma \ref{lemS2x} and Lemma \ref{nS2S3S5}, one has
\begin{align}\label{thm1k6}
\int_{\mathfrak{n}} & \big|S_2^2(\alpha)S_3^2(\alpha)S_6^4(\alpha)\big|d\alpha
\notag
\\ \ll & \sup_{\alpha\in\mathfrak{n}}|S_3(\alpha)|^{\frac{2}{3}}
\Big(\int_0^1\big|S_2^2(\alpha)S_6^8(\alpha)\big|d\alpha\Big)^{\frac{1}{3}}
\Big(\int_0^1\big|S_2^2(\alpha))S_3^2(\alpha)S_6^2(\alpha)\big|d\alpha\Big)^{\frac{2}{3}}
\notag
\\ \ll &  N^{1+\frac{1}{3}-\frac{1}{54}+\varepsilon}.
\end{align}

For $k=7$, by Lemma \ref{lemS2x} and Lemma \ref{nS2S3S5}, we have
\begin{align}\label{thm1k7}
\int_{\mathfrak{n}} & \big|S_2^2(\alpha)S_3^2(\alpha)S_6^2(\alpha)S_7^2(\alpha)\big|d\alpha
\notag
\\ \ll & \sup_{\alpha\in\mathfrak{n}}|S_3(\alpha)|^{\frac{4}{9}}
\Big(\int_0^1\big|S_2^2(\alpha)S_6^8(\alpha)\big|d\alpha\Big)^{\frac{1}{18}}
\Big(\int_0^1\big|S_2^2(\alpha)S_7^{12}(\alpha)\big|d\alpha\Big)^{\frac{1}{6}}
\notag
\\ &\times  \Big(\int_0^1\big|S_2^2(\alpha)S_3^2(\alpha)S_6^2(\alpha)\big|d\alpha\Big)^{\frac{7}{9}}
\notag
\\ \ll &  N^{1+\frac{2}{7}-\frac{1}{81}+\varepsilon}.
\end{align}

For $k\ge 8$ and $x$ in the form (\ref{236kx}), by Lemma \ref{lemS2x} and Lemma \ref{nS2S3S5}, we have
\begin{align}\label{thm1k8}
\int_{\mathfrak{n}} & \big|S_2^2(\alpha)S_3^2(\alpha)S_6^2(\alpha)S_k^2(\alpha)\big|d\alpha
\notag
\\ \ll & \sup_{\alpha\in\mathfrak{n}}|S_3(\alpha)|^{\frac{2}{3x}}
\Big(\int_0^1\big|S_2^2(\alpha)S_6^8(\alpha)\big|d\alpha\Big)^{\frac{1}{3x}}
\Big(\int_0^1\big|S_2^2(\alpha)S_3^{2}(\alpha)S_k^{2x}(\alpha)\big|d\alpha\Big)^{\frac{1}{x}}
\notag
\\ &\times \Big(\int_0^1\big|S_2^2(\alpha)S_3^2(\alpha)S_6^2(\alpha)\big|d\alpha\Big)^{1-\frac{4}{3x}}
\notag
\\ \ll &  N^{1+\frac{2}{k}-\frac{1}{54x}+\varepsilon}.
\end{align}
By (\ref{thm1k6}), (\ref{thm1k7}) and (\ref{thm1k8}), we have
\begin{align}\label{thm1236kn}
\int_{\mathfrak{n}} & \big|S_2^2(\alpha)S_3^2(\alpha)S_6^2(\alpha)S_k^2(\alpha)\big|d\alpha\ll N^{1+\frac{2}{k}-\theta_1(k)+\varepsilon}.
\end{align}
Thus, it establishes (\ref{thm1236k}) by (\ref{thm1236km}) and (\ref{thm1236kn}). Hence, Theorem \ref{Thm236k} holds.

\vskip9mm

\section{Proof of Theorem 1.2}

Suppose that $N$ is a large positive integer. Let $S_k(\alpha)$ be defined as in Lemma \ref{LemHuaBou}. Let
$$
r(n)=\sum_{\substack{p_1^2+p_2^4+p_3^4+p_4^k=n\\ N/4< p_1^2, p_2^4, p_3^4, p_4^k\le N}}
(\log p_1)(\log p_2)(\log p_3)(\log p_4),
$$
where $p_1$, $p_2$, $p_3$, $p_4$ are primes
and the major arcs $\mathfrak{M}$, minor arcs $\mathfrak{m}$, $\mathfrak{N}$ and $\mathfrak{n}$ be defined as in section 3.
Then
the weighted number of representations of $n$ in the form of (\ref{244k}) equals
$$r(n)=\int_{0}^{1}S_2(\alpha)S_4^2(\alpha)S_k(\alpha)e(-n\alpha)d\alpha=\int_{\mathfrak{M}}+\int_{\mathfrak{m}}.$$
Whenever $\mathfrak{B}\subset [0,1]$ is measurable, we put
$$r_{\mathfrak{B}}(n, N)=\int_{\mathfrak{B}}S_2(\alpha)S_4^2(\alpha)S_k(\alpha)e(-n\alpha)d\alpha.$$
Then we have
$$r(n)=r_{[0,1]}(n, N)=\int_ 0^1S_2(\alpha)S_4^2(\alpha)S_k(\alpha)e(-n\alpha)d\alpha.$$
Next we will deal with the integral of major arcs and minor arcs respectively. Applying the now standard methods of
enlarging major arcs (c.f. \cite{LiuZhan}), we can get the following result:
\begin{lemma}\label{888q}
For all even integer $n$ with $N<n \le 2N$, one has $r_{\mathfrak{M}}(n, N)\gg N^{\frac{1}{k}}$.
\end{lemma}
\begin{lemma}\label{l1501}
For $\alpha\in \mathfrak{M}(2K)\backslash\mathfrak{M}(K)$, $N^{\frac{2}{5k}}\ll K\ll N^{\frac{1}{8}}$, one has
\begin{align*}
&S_2(\alpha)\ll N^{\frac{1}{2}+\varepsilon}K^{-\frac{1}{2}},\\
&S_3(\alpha)\ll N^{\frac{1}{3}+\varepsilon}K^{-\frac{1}{2}},\\
&S_4(\alpha)\ll N^{\frac{11}{80}+\varepsilon}K^{\frac{1}{2}}+N^{\frac{1}{4}+\varepsilon}K^{-\frac{1}{2}},\\
&S_5(\alpha)\ll N^{\frac{11}{100}+\varepsilon}K^{\frac{1}{2}}+N^{\frac{1}{5}+\varepsilon}K^{-\frac{1}{2}}.
\end{align*}
\end{lemma}
\begin{proof}
The Theorem 2 in \cite{Kumchev} implies that, if $1\le q\le H, \quad (a,q)=1, \quad |q\alpha-a|<HN^{-1}$
with $H\ll N^{\frac{1}{k}}$, then
\begin{align}\label{lemKumchev}
\sum_{p\sim N^{\frac{1}{k}}}e(\alpha p^k)\ll H^{\frac{1}{2}}N^{\frac{11}{20k}+\varepsilon}+\frac{N^{\frac{1}{k}+\varepsilon}}{(q+N|q\alpha-a|)^{\frac{1}{2}}}.
\end{align}
If $\alpha\in\mathfrak{M}(2K)\backslash \mathfrak{M}(K),\textrm{  }\ N^{\frac{2}{5k}}\ll K\ll N^{\frac{1}{8}}$, then this
lemma clearly follows by (\ref{lemKumchev}).
\end{proof}

\noindent\textit{Proof of Theorem \ref{Thm244k}.}
By Bessel's inequality, we have
\begin{align*}
\sum_{N< n \le 2N} & \Big|\int_{\mathfrak{m}}S_2(\alpha)S_4^2(\alpha)S_k(\alpha)e(-n\alpha)d\alpha\Big|^2
\le
\int_{\mathfrak{m}}\big|S_2^2(\alpha)S_4^4(\alpha)S_k^2(\alpha)\big|d\alpha.
\end{align*}
Thus, to prove Theorem \ref{Thm244k}, it suffices to show that
\begin{align}\label{thm1244k}
\int_{\mathfrak{m}}\big|S_2^2(\alpha)S_4^4(\alpha)S_k^2(\alpha)\big|d\alpha
\ll N^{1+\frac{2}{k}-\theta_2(k)+\varepsilon},
\end{align}
where
$\theta_2(k)$ is defined in Theorem \ref{Thm244k}.

Obviously, we know that
\begin{align*}
\int_{\mathfrak{m}}\big|S_2^2(\alpha)S_4^4(\alpha)S_k^2(\alpha)\big|d\alpha\ll &
\int_{\mathfrak{N}\backslash\mathfrak{M}}\big|S_2^2(\alpha)S_4^4(\alpha)S_k^2(\alpha)\big|d\alpha
\\
 &+ \int_{\mathfrak{n}}\big|S_2^2(\alpha)S_4^4(\alpha)S_k^2(\alpha)\big|d\alpha.
\end{align*}
First, we show that
\begin{align}\label{s1506}
\int_{\mathfrak{N}\backslash\mathfrak{M}}\big|S_2^2(\alpha)S_4^4(\alpha)S_k^2(\alpha)\big|d\alpha\ll
N^{1+\frac{2}{k}-\theta_2(k)+\varepsilon}.
\end{align}
It suffices to prove that
$$\int_{\mathfrak{M}(2K)\backslash\mathfrak{M}(K)}\big|S_2^2(\alpha)S_4^4(\alpha)S_k^2(\alpha)\big|d\alpha\ll
N^{1+\frac{2}{k}-\theta_2(k)+\varepsilon}
$$
for $N^{\frac{2}{5k}}\ll K\ll N^{\frac{1}{8}}$. By Lemma \ref{l1501} and Lemma 5.2 in \cite{HoffmanYu}, we have
\begin{align*}\label{}
\int_{\mathfrak{M}(2K)\backslash \mathfrak{M}(K)} & |  S_2^2(\alpha)S_4^4(\alpha)S_k^{2}(\alpha)|d\alpha
\notag
\\
\ll &
\sup_{\alpha\in \mathfrak{M}(2K)\backslash \mathfrak{M}(K)}|S_2^2(\alpha)S_4^4(\alpha)|
\int_{\mathfrak{M}(2K)}\big|S_k^2(\alpha)\big|d\alpha
\notag
\\ \ll & \frac{N^{1+\varepsilon}}{K}\Big(K^2N^{\frac{11}{20}+\varepsilon}+N^{1+\varepsilon}K^{-2}\Big)
\Big(N^{-1}K(N^{\frac{1}{k}}K+N^{\frac{2}{k}})\Big)
\notag
\\ \ll & N^{\frac{11}{20}+\frac{1}{k}+\varepsilon}K^3+N^{\frac{11}{20}+\frac{2}{k}+\varepsilon}K^2+N^{1+\frac{1}{k}+\varepsilon}K^{-1}+N^{1+\frac{2}{k}+\varepsilon}K^{-2}
\notag
\\ \ll & N^{1+\frac{2}{k}-\frac{4}{5k}+\varepsilon}
\notag
\\ \ll & N^{1+\frac{2}{k}-\theta_2(k)+\varepsilon},
\end{align*}
since $\frac{4}{5k}>\theta_2(k)$ for all $k\ge 4$.

Next we show that
\begin{align}\label{s1505}
 \int_{\mathfrak{n}}\big|S_2^2(\alpha)S_4^4(\alpha)S_k^2(\alpha)\big|d\alpha\ll
 N^{1+\frac{2}{k}-\theta_2(k)+\varepsilon}.
\end{align}
For $k\ge 9$ and $x$ in the form in Theorem \ref{Thm244k}, by Lemma \ref{lemS2x} and Lemma \ref{lemS24x}, one has
\begin{align*}\label{}
\int_{\mathfrak{n}} & \big|S_2^2(\alpha)S_4^4(\alpha)S_k^2(\alpha)\big|d\alpha
\notag
\\ \ll &
\Big(\int_0^1|S_2^2(\alpha)S_4^{4}(\alpha)|d\alpha\Big)^{1-\frac{1}{x}}
\Big(\int_0^1|S_2^2(\alpha)S_4^2(\alpha)S_k^{2x}|d\alpha\Big)^{\frac{1}{x}}\sup_{\alpha\in \mathfrak{m}}|S_4(\alpha)|^{\frac{2}{x}}
\notag
\\ \ll & N^{1+\frac{2}{k}-\frac{1}{48x}+\varepsilon}.
\end{align*}
We use Lemma 3.1 of Zhao \cite{Zhao1} to prove (\ref{s1505}) for $4\le k\le 8$, since the methods are same, we
only give the proof for $k=5$ for simplicity.

For $k=5$, by Lemma 3.1 in \cite{Zhao1} with $g(\alpha)=S_4(\alpha)$ and $h(\alpha)=S_5(\alpha)$, one has
\begin{align}\label{871}
\int_{\mathfrak{n}} & \big|S_2^2(\alpha)S_4^4(\alpha)S_5^2(\alpha)\big|d\alpha
\notag
\\ \ll &
 N^{\frac{1}{4}}J_0^{\frac{1}{4}}\Big(\int_{\mathfrak{n}}\big|S_2^4(\alpha)S_4^6(\alpha)S_5^2(\alpha)\big|d\alpha\Big)^{\frac{1}{4}}
\Big(\int_{\mathfrak{n}}\big|S_2^2(\alpha)S_4^3(\alpha)S_5^2(\alpha)\big|d\alpha\Big)^{\frac{1}{2}}
\notag
\\  &+ N^{\frac{1}{4}(1-2^{-4})+\varepsilon}
\int_{\mathfrak{n}}\big|S_2^2(\alpha)S_4^3(\alpha)S_5^2(\alpha)\big|d\alpha,
\end{align}
where
$$J_0\ll N^{-\frac{3}{5}+\varepsilon}$$
by Lemma 2.2 in \cite{Zhao1}.

By H\"older's inequality, Lemmas \ref{lemS2x} and \ref{lemS23k}, one has
\begin{align}\label{872}
\int_{\mathfrak{n}} &\big|S_2^2(\alpha)S_4^3(\alpha)S_5^2(\alpha)\big|d\alpha
\notag
\\ \ll &
 \Big(\int_{\mathfrak{n}}\big|S_2^2(\alpha)S_4^4(\alpha)S_5^2(\alpha)\big|d\alpha\Big)^{\frac{1}{4}}
\Big(\int_0^1\big|S_2^2(\alpha)S_4^4(\alpha)\big|d\alpha\Big)^{\frac{1}{2}}
\Big(\int_0^1\big|S_2^2(\alpha)S_5^6(\alpha)\big|d\alpha\Big)^{\frac{1}{4}}
\notag
\\ \ll & N^{\frac{4}{5}+\varepsilon}
\Big(\int_{\mathfrak{n}}\big|S_2^2(\alpha)S_4^4(\alpha)S_5^2(\alpha)\big|d\alpha\Big)^{\frac{1}{4}}.
\end{align}
By Lemma \ref{nS2S3S5},
\begin{align}\label{873}
\int_{\mathfrak{n}}\big|S_2^4(\alpha)S_4^6(\alpha)S_5^2(\alpha)\big|d\alpha\ll &
 \sup_{\alpha\in \mathfrak{n}}\big|S_2(\alpha)S_4(\alpha)\big|^2\int_{\mathfrak{n}}\big|S_2^2(\alpha)S_4^4(\alpha)S_5^2(\alpha)\big|d\alpha
 \notag
 \\ \ll & N^{\frac{65}{48}+\varepsilon}
\int_{\mathfrak{n}}\big|S_2^2(\alpha)S_4^4(\alpha)S_5^2(\alpha)\big|d\alpha.
\end{align}

By (\ref{871}),(\ref{872}) and (\ref{873}), we have
\begin{align*}
\int_{\mathfrak{n}}\big|S_2^2(\alpha)S_4^4(\alpha)S_5^2(\alpha)\big|d\alpha\ll
 N^{1+\frac{2}{5}-\frac{1}{48}+\varepsilon}.
\end{align*}
Thus, it establishes (\ref{thm1244k}) by (\ref{s1506}) and (\ref{s1505}).

\vskip9mm

\section{Proof of Theorem 1.3}

Suppose that $N$ is a large positive integer. Let $S_k(\alpha)$ be defined as Lemma \ref{LemHuaBou}.
Denote
$$r(n)=\sum_{\substack{p_1^2+p_2^3+p_3^5+p_4^k=n\\ N/4<p_1^2, p_2^3, p_3^5, p_4^k\le N}}
(\log p_1)(\log p_2)(\log p_3)(\log p_4),$$
and the major arcs $\mathfrak{M}$, minor arcs $\mathfrak{m}$, $\mathfrak{N}$ and $\mathfrak{n}$ be defined as in section 3.
Then
the weighted number of representations of $n$ in the form of (\ref{235k}) equals
$$r(n)=\int_{0}^{1}S_2(\alpha)S_3(\alpha)S_5(\alpha)S_k(\alpha)e(-n\alpha)d\alpha=\int_{\mathfrak{M}}+\int_{\mathfrak{m}}.$$
Whenever $\mathfrak{B}\subset [0,1]$ is measurable, we put
$$r_{\mathfrak{B}}(n, N)=\int_{\mathfrak{B}}S_2(\alpha)S_3(\alpha)S_5(\alpha)S_k(\alpha)e(-n\alpha)d\alpha.$$
Then we have
$$r(n)=r_{[0,1]}(n, N)=\int_ 0^1S_2(\alpha)S_3(\alpha)S_5(\alpha)S_k(\alpha)e(-n\alpha)d\alpha.$$
Next we will deal with the integral of major arcs and minor arcs respectively. Applying the now standard methods of
enlarging major arcs (c.f. \cite{LiuZhan}), we can get the following result:
\begin{lemma}
For all even integer $n$ with $N<n \le 2N$, one has $r_{\mathfrak{M}}(n, N)\gg N^{\frac{1}{30}+\frac{1}{k}}$.
\end{lemma}

\noindent\textit{Proof of Theorem \ref{Thm235k}.}
By Bessel's inequality, we have
\begin{align*}
\sum_{N< n \le 2N} & \Big|\int_{\mathfrak{m}}S_2(\alpha)S_3(\alpha)S_5(\alpha)S_k(\alpha)e(-n\alpha)d\alpha\Big|^2
\le
\int_{\mathfrak{m}}\big|S_2^2(\alpha)S_3^2(\alpha)S_5^2(\alpha)S_k^2(\alpha)\big|d\alpha.
\end{align*}
Thus, to prove Theorem \ref{Thm235k}, it suffices to show that
\begin{align}\label{s1507}
\int_{\mathfrak{m}}\big|S_2^2(\alpha)S_3^2(\alpha)S_5^2(\alpha)S_k^2(\alpha)\big|d\alpha\ll
N^{\frac{16}{15}+\frac{2}{k}-\theta_3(k)+\varepsilon},
\end{align}
where
$\theta_3(k)$ is defined in Theorem \ref{Thm235k}. Obviously, one has
\begin{align*}
\int_{\mathfrak{m}}\big|S_2^2(\alpha)S_3^2(\alpha)S_5^2(\alpha)S_k^2(\alpha)\big|d\alpha\ll & \int_{\mathfrak{N}\backslash\mathfrak{M}}\big|S_2^2(\alpha)S_3^2(\alpha)S_5^2(\alpha)S_k^2(\alpha)\big|d\alpha
\\ & +\int_{\mathfrak{n}}\big|S_2^2(\alpha)S_3^2(\alpha)S_5^2(\alpha)S_k^2(\alpha)\big|d\alpha.
\end{align*}
First, we show that
\begin{align}\label{s1508}
\int_{\mathfrak{N}\backslash\mathfrak{M}}\big|S_2^2(\alpha)S_3^2(\alpha)S_5^2(\alpha)S_k^2(\alpha)\big|d\alpha
\ll N^{\frac{16}{15}+\frac{2}{k}-\theta_3(k)+\varepsilon}.
\end{align}
It suffices to prove that
$$\int_{\mathfrak{M}(2K)\backslash\mathfrak{M}(K)}\big|S_2^2(\alpha)S_3^2(\alpha)S_5^2(\alpha)S_k^2(\alpha)\big|d\alpha
\ll N^{\frac{16}{15}+\frac{2}{k}-\theta_3(k)+\varepsilon}$$
with $N^{\frac{1}{5k}}\ll K\ll N^{\frac{1}{8}}$.
By Lemma \ref{l1501} and Lemma 5.2 in \cite{HoffmanYu}, one has
\begin{align*}
\int_{\mathfrak{M}(2K)\backslash\mathfrak{M}(K)} & \big|S_2^2(\alpha)S_3^2(\alpha)S_5^2(\alpha)S_k^2(\alpha)\big|d\alpha
\\
\ll &\frac{ N^{1+\varepsilon}}{K}\frac{N^{\frac{2}{3}+\varepsilon}}{K}(N^{\frac{11}{50}+\varepsilon}K+N^{\frac{2}{5}+\varepsilon}K^{-1})
\Big(N^{-1}K(N^{\frac{1}{k}}K+N^{\frac{2}{k}})\Big)
\\
\ll & N^{\frac{16}{15}+\frac{2}{k}-\frac{4}{5k}+\varepsilon}
\\
\ll & N^{\frac{16}{15}+\frac{2}{k}-\theta_3(k)+\varepsilon},
\end{align*}
since $\frac{4}{5k}>\theta_3(k)$ for all $k\ge 5$. This establishes (\ref{s1508}).

Next, we show that
\begin{align}\label{s1509}
\int_{\mathfrak{n}} & \big|S_2^2(\alpha)S_3^2(\alpha)S_5^2(\alpha)S_k^2(\alpha)\big|d\alpha
\ll N^{\frac{16}{15}+\frac{2}{k}-\theta_3(k)+\varepsilon}.
\end{align}
For $k=5$, by Lemma 3.1 in \cite{Zhao1} with $g(\alpha)=S_3(\alpha)$ and $h(\alpha)=S_5(\alpha)$, one has
\begin{align}\label{x01}
\int_{\mathfrak{n}} & \big|S_2^2(\alpha)S_3^2(\alpha)S_5^4(\alpha)\big|d\alpha
\notag
\\ \ll &
N^{\frac{1}{3}}J_1^{\frac{1}{4}}\Big(\int_0^1\big|S_2^4(\alpha)S_3^2(\alpha)S_5^6(\alpha)\big|d\alpha\Big)^{\frac{1}{4}}
\Big(\int_0^1\big|S_2^2(\alpha)S_3(\alpha)S_5^4(\alpha)\big|d\alpha\Big)^{\frac{1}{2}}
\notag
\\ & + N^{\frac{1}{3}(1-2^{-3})}\int_0^1\big|S_2^2(\alpha)S_3(\alpha)S_5^4(\alpha)\big|d\alpha,
\end{align}
where
$$J_1\ll N^{-\frac{3}{5}+\varepsilon}$$
by Lemma 2.2 in \cite{Zhao1}.
By H\"older's inequality, Lemmas \ref{lemS2x} and \ref{lemS23k}, one has
\begin{align}\label{x02}
\int_{\mathfrak{n}}  \big|S_2^2(\alpha)S_3(\alpha)S_5^4(\alpha)\big|d\alpha
\notag
\ll &
\Big(\int_0^1\big|S_2^2(\alpha)S_3^2(\alpha)S_5^2(\alpha)\big|d\alpha\Big)^{\frac{1}{2}}
\Big(\int_0^1\big|S_2^2(\alpha)S_5^6(\alpha)\big|d\alpha\Big)^{\frac{1}{2}}
\notag
\\ \ll & N^{\frac{17}{15}+\varepsilon}.
\end{align}
Also,
\begin{align}\label{x03}
\int_{\mathfrak{n}}  \big|S_2^4(\alpha)S_3^2(\alpha)S_5^6(\alpha)\big|d\alpha
\notag
\ll &
\sup_{\alpha\in \mathfrak{n}}\big|S_2^2(\alpha)S_5^2(\alpha)\big|\int_0^1\big|S_2^2(\alpha)S_3^2(\alpha)S_5^4(\alpha)\big|d\alpha
\notag
\\ \ll & N^{\frac{19}{15}+\varepsilon}\int_0^1\big|S_2^2(\alpha)S_3^2(\alpha)S_5^4(\alpha)\big|d\alpha.
\end{align}
Thus,by (\ref{x01})-(\ref{x03}), we obtain that
\begin{align*}
\int_{\mathfrak{n}}  \big|S_2^2(\alpha)S_3^2(\alpha)S_5^4(\alpha)\big|d\alpha
\ll N^{1+\frac{7}{15}-\frac{1}{24}+\varepsilon}.
\end{align*}
It establishes (\ref{s1509}) for $k=5$.

For $k=6$, applying H\"older's inequality, Lemmas \ref{lemS2x} and \ref{lemS23k}, one has
\begin{align*}\label{}
\int_{\mathfrak{n}} & \big|S_2^2(\alpha)S_3^2(\alpha)S_5^{2}(\alpha)S_6^2(\alpha)\big|d\alpha
\notag
\\ \ll & \sup_{\alpha\in \mathfrak{n}}\big|S_3(\alpha)\big|^{\frac{8}{9}}\Big(\int_0^1\big|S_2^2(\alpha)S_5^6(\alpha)\big|d\alpha\Big)^{\frac{1}{3}}
\Big(\int_0^1\big|S_2^2(\alpha)S_6^8(\alpha)\big|d\alpha\Big)^{\frac{1}{9}}
\notag
\\ &  \times\Big(\int_0^1\big|S_2^2(\alpha)S_3^2(\alpha)S_6^2(\alpha)\big|d\alpha\Big)^{\frac{5}{9}}
\notag
\\ \ll  & N^{1+\frac{2}{5}-\frac{2}{81}+\varepsilon}.
\end{align*}
It establishes (\ref{s1509}) for $k=6$.

For $k\ge 7$ and $x$ in the form in Theorem \ref{Thm235k}, by Lemmas \ref{lemS2x}, \ref{lemS23k} and \ref{nS2S3S5}, one has
\begin{align*}
\int_{\mathfrak{n}} & \big|S_2^2(\alpha)S_3^2(\alpha)S_5^{2}(\alpha)S_k^2(\alpha)\big|d\alpha
\notag
\\ \ll & \sup_{\alpha\in \mathfrak{n}}\big|S_3(\alpha)\big|^{\frac{1}{x}}\Big(\int_0^1\big|S_2^2(\alpha)S_5^6(\alpha)\big|d\alpha\Big)^{\frac{1}{2x}}
\Big(\int_0^1\big|S_2^2(\alpha)S_3^2(\alpha)S_5^2(\alpha)\big|d\alpha\Big)^{1-\frac{3}{2x}}
\notag
\\ \times & \Big(\int_0^1\big|S_2^2(\alpha)S_3^2(\alpha)S_k^{2x}(\alpha)\big|d\alpha\Big)^{\frac{1}{x}}
\notag
\\ \ll  & N^{\frac{16}{15}+\frac{2}{k}-\frac{1}{36x}+\varepsilon}.
\end{align*}
It establishes (\ref{s1509}) for $k\ge 7$. Hence, (\ref{s1507}) holds by (\ref{s1508}) and (\ref{s1509}),
and it establishes
Theorem \ref{Thm235k}.

\vskip9mm

\section{Proof of Theorem \ref{Thm3333k}}
Suppose that $N$ is a large positive integer. Let $S_k(\alpha)$ be defined as in Lemma \ref{LemHuaBou}. Let
$$r(n)=\sum_{\substack{p_1^3+p_2^3+p_3^3+p_4^3+p_5^k=n\\ N/5< p_1^3, p_2^3, p_3^3, p_4^3, p_5^k\le N}}
(\log p_1)(\log p_2)(\log p_3)(\log p_4)(\log p_5),$$
and the major arcs $\mathfrak{M}=\mathfrak{M}(Q)$, minor arcs $\mathfrak{m}$, $\mathfrak{N}$ and $\mathfrak{n}$ be defined as in section 3
with $Q=N^{\frac{1}{2k}}$.
Then
the weighted number of representations of $n$ in the form of (\ref{3333k}) equals
$$r(n)=\int_{0}^{1}S_3^4(\alpha)S_k(\alpha)e(-n\alpha)d\alpha=\int_{\mathfrak{M}}+\int_{\mathfrak{m}}.$$
Whenever $\mathfrak{B}\subset [0,1]$ is measurable, we put
$$r_{\mathfrak{B}}(n, N)=\int_{\mathfrak{B}}S_3^4(\alpha)S_k(\alpha)e(-n\alpha)d\alpha.$$
Then we have
$$r(n)=r_{[0,1]}(n, N)=\int_ 0^1S_3^4(\alpha)S_k(\alpha)e(-n\alpha)d\alpha.$$
Next we will deal with the integral of major arcs and minor arcs respectively. Applying the now standard methods of
enlarging major arcs (c.f. \cite{LiuZhan}), we can get the following result:
\begin{lemma}
For all odd integer $n$ with $N<n\le 2N$, one has $r_{\mathfrak{M}}(n, N)\gg N^{\frac{1}{k}+\frac{1}{3}}$.
\end{lemma}

\noindent\textit{Proof of Theorem \ref{Thm3333k}.}
By Bessel's inequality, we have
\begin{align*}
\sum_{N< n \le 2N} & \Big|\int_{\mathfrak{m}}S_3^4(\alpha)S_k(\alpha)e(-n\alpha)d\alpha\Big|^2
\le
\int_{\mathfrak{m}}\big|S_3^8(\alpha)S_k^2(\alpha)\big|d\alpha.
\end{align*}
Thus, to prove Theorem \ref{Thm3333k}, it suffices to show that
\begin{align}\label{s1701}
\int_{\mathfrak{m}}\big|S_3^8(\alpha)S_k^2(\alpha)\big|d\alpha\ll
N^{\frac{5}{3}+\frac{2}{k}-\theta_4(k)+\varepsilon},
\end{align}
where $\theta_4(k)$ is defined in Theorem \ref{Thm3333k}. Obviously, one has
\begin{align*}
\int_{\mathfrak{m}}\big|S_3^8(\alpha)S_k^2(\alpha)\big|d\alpha\ll  \int_{\mathfrak{N}\backslash\mathfrak{M}}\big|S_3^8(\alpha)S_k^2(\alpha)\big|d\alpha
+ \int_{\mathfrak{n}}\big|S_3^8(\alpha)S_k^2(\alpha)\big|d\alpha.
\end{align*}
First, we show that
\begin{align}\label{s1702}
\int_{\mathfrak{N}\backslash\mathfrak{M}}\big|S_3^8(\alpha)S_k^2(\alpha)\big|d\alpha
\ll N^{\frac{5}{3}+\frac{2}{k}-\theta_4(k)+\varepsilon}.
\end{align}
It suffices to prove that
$$\int_{\mathfrak{M}(2K)\backslash\mathfrak{M}(K)}\big|S_3^8(\alpha)S_k^2(\alpha)\big|d\alpha
\ll N^{\frac{5}{3}+\frac{2}{k}-\theta_4(k)+\varepsilon}$$
with $N^{\frac{1}{2k}}\ll K\ll N^{\frac{1}{8}}$.
By Lemma \ref{l1501} and Lemma 5.2 in \cite{HoffmanYu}, one has
\begin{align*}
\int_{\mathfrak{M}(2K)\backslash\mathfrak{M}(K)}\big|S_3^8(\alpha)S_k^2(\alpha)\big|d\alpha
\ll & N^{\frac{8}{3}+\varepsilon}K^{-4}
\Big(N^{-1}K(N^{\frac{1}{k}}K+N^{\frac{2}{k}})\Big)
\\
\ll & N^{\frac{5}{3}+\frac{2}{k}-\frac{3}{2k}+\varepsilon}
\\
\ll & N^{\frac{5}{3}+\frac{2}{k}-\theta_4(k)+\varepsilon},
\end{align*}
since $\frac{3}{2k}>\theta_4(k)$ for all $k\ge 4$. This establishes (\ref{s1702}).

Next, we show that
\begin{align}\label{s1703}
\int_{\mathfrak{n}} & \big|S_3^8(\alpha)S_k^2(\alpha)\big|d\alpha
\ll N^{\frac{5}{3}+\frac{2}{k}-\theta_4(k)+\varepsilon}.
\end{align}
For $k=4$, by Lemma 3.1 in \cite{Zhao1} with $g(\alpha)=h(\alpha)=S_3(\alpha)$, one has
\begin{align}\label{10310}
\int_{\mathfrak{n}}\big|S_3^8(\alpha)S_4^2(\alpha)\big|d\alpha\ll & N^{\frac{1}{3}}J^{\frac{1}{4}}\Big(\int_{\mathfrak{n}}\big|S_3^{12}(\alpha)S_4^4(\alpha)\big|d\alpha\Big)^{\frac{1}{4}}
\Big(\int_{\mathfrak{n}}\big|S_3^{7}(\alpha)S_4^2(\alpha)\big|d\alpha\Big)^{\frac{1}{2}}
\notag
\\  &+
N^{\frac{1}{3}(1-2^{-3})}\int_{\mathfrak{n}}\big|S_3^{7}(\alpha)S_4^2(\alpha)\big|d\alpha
\end{align}
where $$J\ll N^{-\frac{1}{3}+\varepsilon}$$
by Lemma 2.2 in \cite{Zhao1}.
By H\"older's inequality and Lemma \ref{8233}, one has
\begin{align}\label{10311}
\int_{\mathfrak{n}}\big|S_3^{7}(\alpha)S_4^2(\alpha)\big|d\alpha
&\ll
\Big(\int_0^1\big|S_3^{8}(\alpha)\big|d\alpha\Big)^{\frac{3}{4}}
\Big(\int_0^1\big|S_3^{4}(\alpha)S_4^8(\alpha)\big|d\alpha\Big)^{\frac{1}{4}}\notag\\
&\ll N^{\frac{11}{6}+\varepsilon}.
\end{align}
Also,
\begin{align}\label{10312}
\int_{\mathfrak{n}}\big|S_3^{12}(\alpha)S_4^4(\alpha)\big|d\alpha
&\ll \sup_{\alpha\in \mathfrak{n}}|S_3(\alpha)|^6\Big(\int_0^1\big|S_3^{8}(\alpha)\big|d\alpha\Big)^{\frac{1}{2}}
\Big(\int_0^1\big|S_3^{4}(\alpha)S_4^8(\alpha)\big|d\alpha\Big)^{\frac{1}{2}}\notag\\
&\ll N^{\frac{23}{6}+\varepsilon}.
\end{align}
Thus, by (\ref{10310})-(\ref{10312}), we obtain that
$$\int_{\mathfrak{n}}\big|S_3^8(\alpha)S_4^2(\alpha)\big|d\alpha\ll N^{\frac{5}{3}+\frac{1}{2}-\frac{1}{24}+\varepsilon}.$$
It establishes (\ref{s1703}) for $k=4$.

For $k\ge 5$ and $x$ in the form in Theorem \ref{Thm3333k}, by H\"older's inequality and Lemmas \ref{8233} and \ref{nS2S3S5}, one has
\begin{align*}
\int_{\mathfrak{n}}\big|S_3^8(\alpha)S_k^2(\alpha)\big|d\alpha
\ll &
\sup_{\alpha\in\mathfrak{n}}\big|S_3(\alpha)\big|^{\frac{4}{x}}\big(\int_0^1\big|S_3^8(\alpha)\big|d\alpha\big)^{1-\frac{1}{x}}
\big(\int_0^1\big|S_3^4(\alpha)S_k^{2x}(\alpha)\big|d\alpha\big)^{\frac{1}{x}}
\notag
\\ \ll & N^{\frac{5}{3}+\frac{2}{k}-\theta_4(k)+\varepsilon}.
\end{align*}
This establishes (\ref{s1703}) for $k\ge 5$. Hence, (\ref{s1701}) holds by
(\ref{s1702}) and (\ref{s1703}).Thus it establishes
Theorem \ref{Thm3333k}.

\textbf{Acknowlegements.}
This work is supported by National Natural Foundations of China (No. 11771176) .

\vskip4mm

\vskip9mm

\begin{thebibliography}{2010}

\bibitem{Bauer1} C. Bauer,
{\it An improvement on a theorem of the Goldbach-Waring type}, Rocky Mount. J. Math.
{\bf 31} (2001), 1151--1170.

\bibitem{Bauer2} C. Bauer,
{\it A remark on a theorem of the Goldbach Waring type}, Studia Sci. Math. Hungar.
{\bf 41} (2004), 309--324.

\bibitem{Bourgain} J. Bourgain, {\it On the Vinogradov mean value,},
Trudy Mat. Inst. Steklova. {\bf 296} (2017), 36--46.

\bibitem{Brudern} J. Br\"udern, {\it A ternary problem in addditive prime number theory}, J., Steuding,
R. (eds.)From Arithmetic to Zeta-Functions. Springer, Cham (2016).

\bibitem{HoffmanYu} J. W. Hoffman and G. Yu,
{\it A ternary addtive problem}, Monatsh Math.
{\bf 172} (2013), 293--321.

\bibitem{Hua} L.K. Hua,
{\it Additive theory of prime numbers}, Science press, Beijing 1957; English version, American
Mathematical Society, Providence 1965.

\bibitem{Kumchev} A. Kumchev,
{\it On weyl sums over primes and almost primes}, Michigan Math. J.
{\bf 54} (2006), 243--268.

\bibitem{LiuZhan} J.Y. Liu and T. Zhan,
{\it Sums of five almost equal prime squares (II)}, Sci. in China.
{\bf 41} (1998), 710--722.

\bibitem{LiuZ} Z.X. Liu,
{\it An improvement on Waring-Goldbach problem for unlike powers}, Acta Math. Hungar.
{\bf 130} (2011), 118--139.

\bibitem{LuShan} M.G. Lu and Z. Shan,
{\it A problem of Waring-Goldbach type}, J. China Univ. Sci. Tech.
Suppl. I (1982), 1--8 (in Chinese).

\bibitem{Prachar} K. Prachar,
{\it \"Uber ein Problem vom Waring-Goldbach'schen Typ}, Monatsh. Math.
{\bf 57} (1953), 66--74.

\bibitem{RenTsang1} X.M. Ren and Kai-Man Tsang,
{\it Goldbach-Waring problem for unlike powers}, Acta Math. Sinica, English Series (2).
{\bf 23} (2007), 265--280.

\bibitem{RenTsang2} X. M. Ren and Kai-Man Tsang,
{\it Goldbach-Waring problem for unlike powers (II)}, Acta Math. Sinica (Chinese Series).
{\bf 50} (2007), 175--182.

\bibitem{Schwarz} W. Schwarz,
{\it Zur Darstellung von Zahlen durch Summen von Primzahlpotenzen}, II. J. Reine Angew. Math.
{\bf 206} (1961), 78--112.

\bibitem{Vaughan}
R. C. Vaughan, {\it The Hardy-Littlewood method}, 2nd ed.
Cambridge University Press, Cambridge 1997.

\bibitem{Zhao1} L. Zhao,
{\it On the Waring-Goldbach problem for fourth and sixth  powers}, Proc. London Math. Soc.
{\bf 108} (2014), 1593--1622.

\bibitem{Zhao2} L. Zhao,
{\it The additive problem with one cube and three cubes of primes}, Michigan Math. J.
{\bf 63} (2014), 763--779.

\bibitem{Zhao3} L. Zhao,
{\it The exceptional set for sums of unlike powers of primes}, Acta Math. Sinica, English Series.
{\bf 30} (2014), 1897--1904.

\end{thebibliography}
\end{document}